\documentclass{amsart}
\usepackage{cleveref,mathrsfs}
\usepackage{color}

\dedicatory{Dedicated to Professor Vesselin Drensky on his 75th birthday.}

\title{Polynomial identities of finite prime Universal algebras}
\author[Y.~Bahturin et al]{Yuri Bahturin}
\thanks{Y.~Bahturin is supported by NSERC Discovery grant \#227060-19}
\address{Department of Mathematics and Statistics, Memorial University of Newfoundland, St. John's, NL, A1C5S7, Canada.}
\email{bahturin@mun.ca}

\author[]{Daniela Martinez Correa}
\thanks{D.~Correa is supported by Fapesp, grant no.~2024/01338-0 and 2025/11605-8.}
\address{Department of Mathematics, Instituto de Matem\'atica e Estat\'istica, Universidade de S\~ao Paulo, SP, Brazil}
\email{danielam.correa@ime.usp.br}

\author[]{Diogo Diniz}
\thanks{D.~Diniz was partially supported by Conselho Nacional de Desenvolvimento Cient\'ifico e Tecnol\'ogico (CNPq) grant No.~304328/2022-7.}
\address{Unidade Acadêmica de Matemática, Universidade Federal de Campina Grande, Campina Grande, PB, 58429-970, Brazil}
\email{diogo@mat.ufcg.edu.br}

\author[]{Felipe Yasumura}
\thanks{F.~Yasumura is supported by Fapesp grant 2024/14914-9.}
\address{Department of Mathematics, Instituto de Matem\'atica, Estat\'istica e Ci\^encia da Computa\c c\~ao, Universidade de S\~ao Paulo, SP, Brazil}
\email{fyyasumura@ime.usp.br}

\newtheorem{Thm}{Theorem}
\newtheorem{lemma}[Thm]{Lemma}
\newtheorem{proposition}[Thm]{Proposition}
\newtheorem{cor}[Thm]{Corollary}
\theoremstyle{definition}
\newtheorem{definition}[Thm]{Definition}
\newtheorem{example}{Example}



\begin{document}
\begin{abstract}
We prove that two finite prime $\Omega$-algebras defined over the same unital commutative ring and satisfying the same set of polynomial identities are isomorphic.
\end{abstract}
\maketitle

\section{Introduction}
An interesting question in PI-theory is whether an algebra can be determined by the set of polynomial identities it satisfies. 
In general, the answer is negative. 
However, under certain conditions, positive results are known; see, for instance, \cite{AH, BD2019, KZ, KR, Neher, SZ}. 
The strongest result in this direction is \cite[Section \S5]{Raz}, where the author proves that finite-dimensional prime \(\Omega\)-algebras over an algebraically closed field are uniquely determined by their polynomial identities 
(see also \cite{BY, BY2} for the graded version). 
It is worth mentioning that, even when not prime, gradings on the algebra of upper triangular matrices are determined by their graded polynomial identities 
(see \cite{VinKoVa2004,GR2020}).

In the context of certain finite simple universal algebras, several isomorphism results have been obtained; see, for instance, \cite{FM,J,HM}. The techniques we use here are more appropriate in the context of linear $\Omega$-algebras and our main isomorphism result deals with more general prime algebras.

The finite basis property, one of the main problems in PI-theory, has been positively solved in several varieties defined by finite algebras and rings. 
For example, L'vov proved, among other relevant results, that finite associative and finite alternative algebras have a finite basis of identities \cite{Lvov, Lvov2}, and the associative case was also treated independently by Kruse \cite{Kruse}; 
Bahturin and Olshanskii solved the problem for finite Lie rings \cite{BO}; 
Medvedev obtained results for finite Jordan algebras \cite{Medvedev1} and for a broad class of varieties \cite{Medvedev2}. 
For finite groups, the finite basis property was proved by Oates-Williams and Powell in 1964 \cite{OaPo} (see \cite[Chapter 5]{Hanna} as well). 
It is worth noting that, in general, the Specht property does not hold for arbitrary finite algebras (see, for instance, \cite{Lvov3, Polin}).

In this paper, we prove that finite prime \(\Omega\)-algebras are uniquely determined by their polynomial identities. 
Moreover, if \(\mathcal{A} \in \mathrm{var}(\mathcal{A}')\) and \(\mathcal{A}\) is finite and prime, then \(\mathcal{A}\) is section of \(\mathcal{A}'\). The techniques used here are inspired by those in group theory (see the theory compiled in \cite[Chapter 5]{Hanna}; see also \cite[Chapter 7]{bahturinbook} and \cite{BO}).

\section{Preliminaries}
We let $\mathbf{k}$ be a unital commutative ring. Let $\Omega=\bigcup_{n\ge0}\Omega_n$. An $\Omega$-algebra is a $\mathbf{k}$-module $\mathcal{A}$ where each $\omega\in\Omega_n$ defines a $\mathbf{k}$-multilinear map
$$
\omega:\underbrace{\mathcal{A}\times\cdots\times\mathcal{A}}_\text{$n$ times}\to\mathcal{A}.
$$
The map defined by each element of $\Omega_0$ is a choice of a distinguished element of $\mathcal{A}$. An ideal is a $\mathbf{k}$-submodule $I\subseteq\mathcal{A}$ such that $\omega(a_1,\ldots,a_n)\in I$, for any $\omega\in\Omega_n$ ($n>0$), and $a_1$, \dots, $a_n\in\mathcal{A}$, where $\{a_1,\ldots,a_n\}\cap I\ne\emptyset$. A homomorphism between two $\Omega$-algebras $\mathcal{A}$ and $\mathcal{B}$ is a $\mathbf{k}$-linear map $f:\mathcal{A}\to\mathcal{B}$ such that $f\omega(a_1,\ldots,a_n)=\omega(f(a_1),\ldots,f(a_n))$, for each $\omega\in\Omega_n$ and $a_1$, \dots, $a_n\in\mathcal{A}$.

The free $\Omega$-algebra, denoted by $\mathbf{k}_\Omega\langle X\rangle$ is constructed as follows. First, we inductively construct the set of ``monomials'' $M_\Omega(X)$. It contains $X$ and $\Omega_0$. Given $\omega\in\Omega_n$ and $y_1$, \dots, $y_n\in M_\Omega(X)$, then $\omega(y_1,\ldots,y_n)\in M_\Omega(X)$. So, $\mathbf{k}_\Omega\langle X\rangle$ is the free $\mathbf{k}$-module with basis $M_\Omega(X)$. It has a natural structure of $\Omega$-algebra, and it is called  the \emph{free $\Omega$-algebra} over $\mathbf{k}$, freely generated by $X$. We have the usual notion of degree and multilinear polynomials. From now on, we let $X=\{x_1,x_2,\ldots\}$ be an infinite set of variables indexed by $\mathbb{N}$.

Now, let $\mathcal{S}_1$, \dots, $\mathcal{S}_m\subseteq\mathcal{A}$ be nonempty subsets of $\mathcal{A}$, where $m\ge2$. We define
\begin{align*}
\mathcal{S}_1\cdot\mathcal{S}_2\cdot\ldots\cdot\mathcal{S}_m=\mathrm{Span}_\mathbf{k}\{&f(s_1,\ldots,s_m,a_1,\ldots,a_t)\mid
f=f(x_1,\ldots,x_{m+t})\in \mathbf{k}_\Omega\langle X\rangle\\%
&\text{ is multilinear}, s_1\in\mathcal{S}_1,\dots,s_m\in\mathcal{S}_m, a_1,\ldots,a_t\in\mathcal{A}\}.
\end{align*}
It is worth noting that $\mathcal{S}_1\cdot\ldots\cdot\mathcal{S}_m$ is always an ideal. For each $k\in\mathbb{N}$, we define $\mathcal{S}^k:=\underbrace{\mathcal{S}\cdot\ldots\cdot\mathcal{S}}_\text{$k$ times}$. Note that, $\mathcal{S}_1\cdot(\mathcal{S}_2\cdot\mathcal{S}_2)$ does not need to coincide with $\mathcal{S}_1\cdot\mathcal{S}_2\cdot\mathcal{S}_3$. However, it is easy to check that $\mathcal{S}_1\cdot(\mathcal{S}_2\cdot\mathcal{S}_2)\subseteq\mathcal{S}_1\cdot\mathcal{S}_2\cdot\mathcal{S}_3$. A similar consideration can be done for arbitrary disposition of parenthesis on a greater number of product of subsets. In particular, we have
$$
\mathcal{A}\cdot\mathcal{A}^k\subseteq\mathcal{A}^{k+1},
$$
for each $k\in\mathbb{N}$.

Recall that an $\Omega$-algebra $\mathcal{A}$ is said to be:
\begin{enumerate}
    \item \emph{semiprime} if, for every nonzero ideal $I \subseteq \mathcal{A}$, we have $I^2 \neq 0$;
    \item \emph{prime} if, for every pair of nonzero ideals $I, J \subseteq \mathcal{A}$, we have $I \cdot J \neq 0$;
    \item \emph{simple} if $\mathcal{A}^2 \neq 0$ and the only ideals of $\mathcal{A}$ are $0$ and $\mathcal{A}$.
\end{enumerate}

A polynomial \(f=f(x_1,\ldots,x_m)\in \mathbf{k}_\Omega\langle X\rangle\) is a polynomial identity 
for an \(\Omega\)-algebra \(\mathcal{A}\) if \(f(a_1,\ldots,a_m)=0\) for all 
\(a_1,\ldots,a_m\in\mathcal{A}\). A PI-algebra is an \(\Omega\)-algebra \(\mathcal{A}\) that satisfies a multilinear polynomial identity where at least one coefficient is $1$. This is always the case when \(\mathcal{A}\) is finite. 
We denote by \(\mathrm{Id}(\mathcal{A})\) the set of all polynomial identities of \(\mathcal{A}\).

A variety of \(\Omega\)-algebras over \(\mathbf{k}\) is a class \(\mathscr{V}\) of \(\Omega\)-algebras over \(\mathbf{k}\) 
defined by a set of polynomial identities. 
A relatively free algebra in \(\mathscr{V}\), or a \(\mathscr{V}\)-free algebra, is a free algebra 
within the class \(\mathscr{V}\). 
Given a PI-algebra \(\mathcal{A}\), we denote by \(\mathrm{var}(\mathcal{A})\) the variety 
of \(\Omega\)-algebras over \(\mathbf{k}\) determined by \(\mathrm{Id}(\mathcal{A})\).

Let \(\mathcal{A}\) be a PI-algebra. We shall construct a model for the relatively free 
algebra in \(\mathrm{var}(\mathcal{A})\). 
Let \(Z\) be any set, and let \(\mathcal{A}^{(Z)}\) denote the set of all maps 
\(Z\to\mathcal{A}\) with finite support. 
Set \(\mathcal{A}^\ast := \mathcal{A}^{\mathcal{A}^{(Z)}}\). 
For each \(z\in Z\), define \(h_z \in \mathcal{A}^\ast\) by
\[
h_z(f) = f(z), \qquad f\in\mathcal{A}^{(Z)}.
\]
Let 
\[
\mathcal{R}_Z := \mathrm{alg}\{\,h_z \mid z\in Z\,\} \subseteq \mathcal{A}^\ast
\]
be the subalgebra of \(\mathcal{A}^\ast\) generated by all elements \(h_z\), \(z\in Z\).

\begin{proposition}\label{free_alg}
The algebra \(\mathcal{R}_Z\) is relatively free in \(\mathrm{var}(\mathcal{A})\), freely generated by \(Z\).
\end{proposition}

\begin{proof}
The proof follows the same argument as in \cite[Theorem 15.4]{Hanna}.
\end{proof}

As a consequence, if \(\mathcal{A}\) is finite, then \(\mathcal{R}_Z\) is finite whenever 
\(Z\) is finite. In particular, \(\mathrm{var}(\mathcal{A})\) is locally finite.

\section{Main result}
We let \(\mathcal{A}\) be an \(\Omega\)-algebra over an arbitrary unital commutative ring \(\mathbf{k}\), 
and assume that \(\Omega_n \ne \emptyset\) for at least one \(n > 1\).

\begin{definition}
A \emph{section} of \(\mathcal{A}\) is an algebra of the form 
\(\mathcal{R}/\mathcal{I}\), where \(\mathcal{R} \subseteq \mathcal{A}\) is a subalgebra 
and \(\mathcal{I} \subseteq \mathcal{R}\) is an ideal of \(\mathcal{R}\).
\end{definition}

\begin{definition}
Let \(\mathcal{A}\) be an \(\Omega\)-algebra. Its \emph{monolith}, denoted \(\mathrm{M}(\mathcal{A})\), 
is the intersection of all its nonzero ideals. 
We say that \(\mathcal{A}\) is \emph{monolithic} if \(\mathrm{M}(\mathcal{A}) \ne 0\).
\end{definition}

\begin{lemma}\label{prime_monolith}
Let \(\mathcal{R}\) be a finite \(\Omega\)-algebra. 
Then \(\mathrm{M}(\mathcal{R})^2 \ne 0\) if and only if \(\mathcal{R}\) is prime.
\end{lemma}

\begin{proof}
Assume that \(\mathcal{R}\) is prime. 
Note that for any two nonzero ideals \(I, I'\) in a prime algebra, their intersection is nonzero: 
\[
0 \ne II' \subseteq I \cap I'.
\] 
Hence, the intersection of all nonzero ideals of \(\mathcal{R}\) is a nonzero ideal 
\(\mathrm{M}(\mathcal{R})\). Since \(\mathcal{R}\) is prime, its square is nonzero: \(\mathrm{M}(\mathcal{R})^2 \ne 0\).

Conversely, assume that \(\mathrm{M}(\mathcal{R})^2 \ne 0\). 
Then for any nonzero ideals \(I, I' \subseteq \mathcal{R}\), we have
\[
0 \ne \mathrm{M}(\mathcal{R})^2 \subseteq II',
\] 
which shows that \(\mathcal{R}\) is prime.
\end{proof}

Thus, prime algebras play an analogous role to that of a monolithic group with a non-abelian monolith in the study of the laws of finite groups.

\begin{definition}[Critical algebra]
A \emph{critical algebra} is an algebra \(\mathcal{A}\) that does not belong to the variety generated by its proper sections.
\end{definition}

It is clear that every critical algebra is monolithic. 
Indeed, if the intersection of all nonzero ideals of a critical algebra \(\mathcal{A}\) were \(0\), 
then \(\mathcal{A}\) would belong to the variety generated by all proper quotients of \(\mathcal{A}\), 
contradicting its criticality.

\begin{lemma}
Let \(\mathscr{V}\) be a locally finite variety. 
Then \(\mathscr{V}\) is generated by its critical algebras.
\end{lemma}

\begin{proof}
It is known that every variety is generated by its finitely generated algebras. 
Since \(\mathscr{V}\) is locally finite, it is generated by its finite algebras. Hence, it suffices to show that the variety generated by a finite algebra \(\mathcal{A}\) 
is generated by the critical algebras in \(\mathrm{var}(\mathcal{A})\). 
We proceed by induction on the cardinality of \(\mathcal{A}\).

If \(\mathcal{A} \in \mathscr{V}\) has minimal cardinality, then it has no proper sections, 
and thus \(\mathcal{A}\) is critical. Now, if \(\mathcal{A}\) has larger cardinality, then either \(\mathcal{A}\) is critical, 
or it belongs to the variety generated by its proper sections. 
In the latter case, by the induction hypothesis, each proper section belongs to the variety generated by its critical sections, which are also sections of \(\mathcal{A}\). 
Thus \(\mathcal{A}\) is also generated by critical algebras.
\end{proof}

Let \(\mathscr{V}\) be a locally finite variety defined over a unital commutative ring 
by an arbitrary set of finite algebras \(\mathscr{A}\). We shall assume that $\mathscr{A}$ is \emph{section-closed}, i.e., each section of each algebra in $\mathscr{A}$ is isomorphic to some algebra in $\mathscr{A}$.
From the proof of Birkhoff's theorem, the relatively free algebra of finite rank of 
\(\mathscr{V}\) is a subalgebra of a finite direct product of algebras in \(\mathscr{A}\). 
Thus, each \(\mathcal{A} \in \mathscr{V}\) is a section of a finite product of algebras in \(\mathscr{A}\).
 
Given a representation
\begin{equation}\label{rep}
\mathcal{A} = \mathcal{B}/C, \quad \mathcal{B} \subseteq \mathcal{P} = \prod_{i=1}^n \mathcal{S}_i,
\end{equation}
we may assume that 
\[
|\mathcal{S}_1| \ge |\mathcal{S}_2| \ge \cdots \ge |\mathcal{S}_n|.
\]

For such a representation, we consider the sequence 
\((|\mathcal{S}_1|, |\mathcal{S}_2|, \ldots, |\mathcal{S}_n|)\) 
and the lexicographical ordering on these sequences. 
A representation is called \emph{minimal} if its sequence is minimal with respect to this ordering.

\begin{lemma}[Minimal representation]\label{minrep}
Let \(\mathcal{A}\) be an \(\Omega\)-algebra in a locally finite variety \(\mathscr{V}\). 
Consider a minimal representation
\[
\mathcal{A} = \mathcal{B}/C, \quad \mathcal{B} \subseteq \mathcal{P} = \prod_{i=1}^n \mathcal{S}_i.
\]
Then:
\begin{enumerate}
\renewcommand{\labelenumi}{(\roman{enumi})}
\item each \(\mathcal{S}_i\) is a critical algebra,
\item \(\mathcal{B}\) is a subdirect product of the \(\mathcal{S}_i\), \(i=1,\ldots,n\),
\item \(\mathcal{D} \subseteq \mathcal{S}_i\) is an ideal of \(\mathcal{S}_i\) if and only if 
      \(\mathcal{B} \mathcal{D} \subseteq \mathcal{D}\),
\item any nontrivial ideal \(\mathcal{D} \subseteq \mathcal{S}_i\) has nontrivial intersection with \(\mathcal{B}\),
\item \(C \cap \mathcal{S}_i = 0\) for each \(i=1,\ldots,n\).
\end{enumerate}
In particular, if $\mathscr{V}=\mathrm{var}(\mathcal{A}')$ then each $\mathcal{S}_i$ is a critical section of $\mathcal{A}'$.
\end{lemma}

\begin{proof}
We prove each property for any minimal representation.

\medskip
\noindent
(i) If some \(\mathcal{S}_i\) is not critical, then it belongs to the variety generated by its proper sections. 
Hence \(\mathcal{S}_i\) is a section of a product of algebras of smaller cardinality, and so is \(\mathcal{A}\), 
contradicting the minimality of the representation.

\medskip
\noindent
(ii) Denote by \(\pi_i: \mathcal{P} \to \mathcal{S}_i\) the projection. 
We may replace \(\mathcal{P}\) by \(\prod_{i=1}^n \pi_i(\mathcal{B})\). 
This would contradict minimality unless \(\pi_i(\mathcal{B}) = \mathcal{S}_i\) for each \(i\). 
In particular, (iii) then holds.

\medskip
\noindent
(iv) Suppose \(0 \ne \mathcal{D} \subseteq \mathcal{S}_1\) has trivial intersection with \(\mathcal{B}\). 
Then we could replace \(\mathcal{P}\) by 
\(\mathcal{S}_1/\mathcal{D} \times \prod_{i=2}^n \mathcal{S}_i\), 
contradicting the minimality of the representation.

\medskip
\noindent
(v) Suppose \(C \cap \mathcal{S}_1 \ne 0\). 
Then we could replace \(\mathcal{P}\) by 
\(\mathcal{S}_1/(C \cap \mathcal{S}_1) \times \prod_{i=2}^n \mathcal{S}_i\), 
again contradicting minimality.
\end{proof}

The next result states that every prime algebra is critical.

\begin{Thm}\label{prime_critical}
Each finite prime \(\Omega\)-algebra \(\mathcal{A}\) is critical. 
In addition, if \(\mathcal{A}\in\mathrm{var}(\mathcal{A}')\), then \(\mathcal{A}\) is a section of \(\mathcal{A}'\).
\end{Thm}

\begin{proof}
Let \(\mathcal{A}\) be a finite prime \(\Omega\)-algebra and write 
\(\mathcal{A} = \mathcal{B}/C\), where 
\(\mathcal{B} \subseteq \mathcal{P} = \prod_{i=1}^n \mathcal{S}_i\), as in \Cref{minrep}. By (iv), set \(I_j := \mathcal{S}_j \cap \mathcal{B}\), which is a nonzero ideal of \(\mathcal{B}\). Moreover, by (v), its projection \(I_j' := I_j/(I_j \cap C)\) is a nonzero ideal of \(\mathcal{B}/C\). If \(n>1\), then \(\mathcal{S}_1 \mathcal{S}_2 = 0\), so that \(I_1' I_2' = 0\), which contradicts the primeness of \(\mathcal{A}\). Hence, \(n=1\).  Again by (v), we have \(C = C \cap \mathcal{S}_1 = 0\), and by (ii), \(\mathcal{B} = \mathcal{S}_1\). 
Thus, \(\mathcal{A} = \mathcal{S}_1\) is critical. 

Finally, if \(\mathcal{A} \in \mathrm{var}(\mathcal{A}')\), we may assume that \(\mathcal{S}_1\) is a critical section of \(\mathcal{A}'\), proving the second statement.
\end{proof}

The first statement is analogous to the result \cite[Corollary 52.34]{Hanna} for finite groups. The second statement is known for certain finite simple algebras (here, algebra means an arbitrary nonempty set with an arbitrary set of $n$-ary operations) (see \cite[Theorem 14.5]{HM}). In the next claim, as usual, we assume that \(\Omega\) contains at least one \(n\)-ary 
operation with \(n\ge 2\). Now an important consequence of Theorem \ref{prime_critical} is the following.

\begin{cor}\label{main_prime}
Let \(\mathcal{A}\) and \(\mathcal{B}\) be finite prime (in particular, simple) 
\(\Omega\)-algebras over the same unital commutative ring $\textbf{k}$, satisfying the same set of polynomial identities. 
Then \(\mathcal{A} \cong \mathcal{B}\).
\end{cor}

\begin{proof}
Since \(\mathcal{A}\) is prime and belongs to \(\mathrm{var}(\mathcal{B})\), 
\Cref{prime_critical} implies that \(\mathcal{A}\) is isomorphic to a section of \(\mathcal{B}\). 
In particular, \(|\mathcal{A}| \le |\mathcal{B}|\). Reversing the argument, \(\mathcal{B}\) is a section of \(\mathcal{A}\). 
This can occur only if \(\mathcal{A} \cong \mathcal{B}\).
\end{proof}

We remark that Theorem \ref{prime_critical} and Corollary \ref{main_prime} are similar to analogous results for finite groups obtained in \cite{KN} (see also \cite[Theorem 53.31]{Hanna}) and \cite[Corollary 53.35]{Hanna}.

Since a finite semigroup grading and its graded polynomial identities can be described 
via a set of unary operations and their polynomial identities (see \cite{BY,BY2}), 
\Cref{main_prime} specializes as follows:

\begin{cor}
Let \(G\) be a semigroup, and let \(\mathcal{A}\) and \(\mathcal{B}\) be 
graded-prime (in particular, graded-simple) finite \(G\)-graded \(\Omega\)-algebras over 
the same unital commutative ring, where \(\Omega\) contains at least one \(n\)-ary 
operation with \(n\ge 2\). 
Then \(\mathrm{Id}_G(\mathcal{A}) = \mathrm{Id}_G(\mathcal{B})\) if and only if \(\mathcal{A} \cong \mathcal{B}\) as \(G\)-graded algebras.\qed
\end{cor}

\section{Similarity}
The following constructions are similar to the classical constructions in the theory of finite groups; we include them here for the sake of completeness.
\begin{definition}
Given an $\Omega$-algebra $\mathcal{A}$ and subsets $\mathcal{B}$, $S\subseteq\mathcal{A}$, we set
\begin{align*}
\mathrm{Ann}_\mathcal{B}(S)=\{x\in\mathcal{B}\mid &\text{for all multilinear $f(y_1,y_2,x_1,\ldots,x_n)\in\mathbf{k}_\Omega\langle X\rangle$},\\&f(x,s,\ldots,b_n)=0,\forall b_1,\ldots,b_n\in\mathcal{B},s\in S\}.
\end{align*}
\end{definition}
If $\mathcal{B}$ is a subalgebra, then it is clear that $\mathrm{Ann}_\mathcal{B}(S)$ is an ideal, and it is the greatest ideal $I\subseteq\mathcal{B}$ such that $S\cdot I=0$. Thus, an algebra $\mathcal{A}$ is prime if and only if $\mathrm{Ann}_\mathcal{A}(I)=0$ for each nonzero ideal $I\subseteq\mathcal{A}$. In addition, a finite algebra $\mathcal{A}$ is prime if and only if $\mathrm{Ann}_\mathcal{A}(\mathrm{M}(\mathcal{A}))=0$.

Now, if $I\subseteq\mathcal{A}$ is an ideal, then $I$ has a structure of $\mathcal{A}$-module. In addition, $I$ has a structure of $\mathcal{A}/\mathrm{Ann}_\mathcal{A}(I)$-module via
$$
f(x,a_1+\mathrm{Ann}_\mathcal{A}(I),\ldots,a_n+\mathrm{Ann}_\mathcal{A}(I)):=f(x,a_1,\ldots,a_n),
$$
for each multilinear $f=f(x_0,x_1,\ldots,x_n)$, $x\in I$ and $a_1$, \ldots, $a_n\in\mathcal{A}$.

\begin{definition}
Let $\mathcal{A}$ and $\mathcal{B}$ be $\Omega$-algebras and $I\subseteq\mathcal{A}$ and $J\subseteq\mathcal{B}$ ideals. We say that $I$ and $J$ are \emph{similar}, denoted by $(I\triangleleft\mathcal{A})\sim(J\triangleleft\mathcal{B})$, if there exist algebra isomorphisms $\alpha:\mathcal{A}/\mathrm{Ann}_\mathcal{A}(I)\to\mathcal{B}/\mathrm{Ann}_\mathcal{B}(J)$ and $\mu:I\to J$ such that
$$
\mu(\omega(a_1,\ldots,a_n))=\omega(\varepsilon_1(a_1),\ldots,\varepsilon_n(a_n)),\quad\omega\in\Omega,
$$
for each $a_1$, \dots, $a_n\in\mathcal{A}/\mathrm{Ann}_\mathcal{A}(I)\cup I$, where $|\{a_1,\ldots,a_n\}\cap I|=1$, and $\varepsilon_i(a_i)=\mu(a_i)$, if $a_i\in I$ and $\varepsilon_i(a_i)=\alpha(a_i)$ otherwise, for each $i$.
\end{definition}
It is clear that similarity is an equivalence relation. Similarity states that not only $I$ and $J$ are isomorphic algebras, but they have isomorphic structures of modules.

We shall need the following result:
\begin{lemma}\label{lem1}
Let $\mathcal{A}\subseteq\mathcal{S}_1\times\mathcal{S}_2$, and assume that $\mathcal{A}$ projects onto $\mathcal{S}_2$. If $I\subseteq\mathcal{A}$ is an ideal contained in $\mathcal{A}\cap\mathcal{S}_2$, then
$$
(I\triangleleft\mathcal{A})\sim(I\triangleleft\mathcal{S}_2).
$$
\end{lemma}
\begin{proof}
We let $\mu:I\to I$ be the identity map. Denote by $\alpha:\mathcal{A}\to\mathcal{S}_2/\mathrm{Ann}_{\mathcal{S}_2}(I)$ the composition of the projection $\mathcal{A}\to\mathcal{S}_2$ with the natural map $\mathcal{S}_2\to\mathcal{S}_2/\mathrm{Ann}_{\mathcal{S}_2}(I)$. It is elementary to check that $\alpha$ factors through $\mathcal{A}/\mathrm{Ann}_\mathcal{A}(I)$. Moreover, $\alpha:\mathcal{A}/\mathrm{Ann}_\mathcal{A}(I)\to\mathcal{S}_2/\mathrm{Ann}_{\mathcal{S}_2}(I)$ is an algebra isomorphism that satisfies the conditions of similarity.
\end{proof}

The following is the main claim of this section.
\begin{lemma}\label{lem2}
If $\mathcal{A}=\mathcal{B}/C$, where $B\subseteq\mathcal{P}=\prod_{i=1}^n\mathcal{S}_i$ is a minimal representation of $\mathcal{A}$, then for each $i$, there exists a minimal ideal $M_i\subseteq\mathcal{A}$ such that
$$
(M_i\triangleleft\mathcal{A})\sim(\mathrm{M}(\mathcal{S}_i)\triangleleft\mathcal{S}_i).
$$
\end{lemma}
\begin{proof}
From the construction of the minimal representation, we see that $M_i:=\mathrm{M}(\mathcal{S}_i)C/C$ is a minimal ideal of $\mathcal{B}/C\cong\mathcal{A}$, isomorphic to $\mathrm{M}(\mathcal{S}_i)$. In addition, \Cref{lem1} gives $(\mathrm{M}(\mathcal{S}_i)\triangleleft\mathcal{S}_i)\sim(M_i\triangleleft\mathcal{B}/C)$.
\end{proof}

In particular, if an algebra is monolithic:
\begin{cor}\label{lem3}
Let $\mathcal{A}=\mathcal{B}/C$, where $B\subseteq\mathcal{P}=\prod_{i=1}^n\mathcal{S}_i$ is a minimal representation of $\mathcal{A}$, and assume that $\mathcal{A}$ is monolithic. Then, $(\mathrm{M}(\mathcal{A})\triangleleft\mathcal{A})\sim(\mathrm{M}(\mathcal{S}_i)\triangleleft\mathcal{S}_i)$, for each $i=1,\ldots,n$.\qed
\end{cor}

\begin{Thm}\label{similarmonoliths}
If two critical algebras $\mathcal{A}$ and $\mathcal{B}$ generate the same variety, then their monoliths are similar.
\end{Thm}

\begin{proof}
Since $\mathcal{A}\in \mathrm{var}(\mathcal{B})$ and $\mathcal{B}$ is critical there exists a minimal representation 
\begin{equation}
\mathcal{A} = \mathcal{B}/C, \quad \mathcal{B} \subseteq \mathcal{P} = \prod_{i=1}^n \mathcal{S}_i,
\end{equation}
where $S_i=B$ for some $i$. Since $\mathcal{A}$ is monolithic the result follows from Corollary \ref{lem3}.
\end{proof}

As a consequence we obtain a slightly improved version of Corollary \ref{main_prime}.

\begin{cor}
If two critical algebras $\mathcal{A}$ and $\mathcal{B}$ generate the same variety and one of them is prime then $\mathcal{A}\cong \mathcal{B}$.
\end{cor}
\begin{proof}
It follows from Theorem \ref{similarmonoliths} and Lemma \ref{prime_monolith} that $\mathcal{A}$ and $\mathcal{B}$ are both prime algebras. The result now follows directly from Corollary \ref{main_prime}.
\end{proof}

\section{Examples}
We present some interesting examples and counter-examples.

\begin{example}\textbf{(Relatively free algebra of a prime algebra is not prime.)}
In contrast with the case where the base ring is an infinite field (see, e.g., \cite[Chapter 12]{DF}), the relatively free algebra 
of the variety generated by a prime algebra need not be prime. 
For example, let \(\mathcal{A} = \mathbf{k}\) be a finite field. 
Then \(\mathcal{A}\) is simple, and in particular, prime. 
The free algebra in \(\mathrm{var}(\mathcal{A})\) of finite rank is commutative and associative. 
Hence, if it were prime, it would be a domain. 
However, finite commutative domains are fields. 
Thus, the relatively free algebra would have to be a field. 
Since \(x^{|\mathcal{A}|}-x\) is a polynomial identity of the relatively free algebra, 
it cannot be an extension field of \(\mathcal{A}\), yielding a contradiction. 
Therefore, the relatively free algebra of \(\mathrm{var}(\mathcal{A})\) of finite rank is not prime. 
\end{example}

\begin{example}\textbf{(Critical not prime algebra.)}
Not every critical algebra is prime. 
For example, let \(\mathbf{k}\) be a finite field of characteristic \(2\) and $\mathcal{L}=\mathfrak{sl}_2(\mathbf{k})$. Note that $\mathcal{L}^2=[\mathcal{L},\mathcal{L}]\ne0$ and $[\mathcal{L},\mathcal{L}^2]=0$. On the other hand, every proper section $\mathcal{H}$ satisfies $\mathcal{H}^2=0$. Hence, $\mathcal{L}$ is critical. Clearly it is not prime.
\end{example}

\begin{example} \textbf{(Monolithic non-critical algebras.)}
Let $\mathbf{k}$ be a field and  $\mathcal{L}=\textnormal{Span}\{x,y,z\}$, be the Heisenberg algebra, i.e., the nonzero products are $xy=-yx=z$. Then, $\mathcal{L}$ is a nilpotent Lie algebra and $\mathfrak{z}(\mathcal{L})=\textrm{Ann}_\mathcal{L}(\mathcal{L})=\textnormal{Span}\{z\}$. Moreover, $\mathcal{L}$ is monolithic.
Now, consider the algebra $\mathcal{L}\times\mathcal{L}$ and notice that $K=\{(a,-a)\mid\ a\in\textnormal{Ann}_{\mathcal{L}}(\mathcal{L})\}$ is an ideal of $\mathcal{L}\times\mathcal{L}$. Then, $\mathcal{L}'=\mathcal{L}\times\mathcal{L}/K$ is a Lie algebra, we shall show that this algebra is monolithic. In fact: Let $J'$ be a nonzero ideal of $\mathcal{L}'$, then $J'=J/K$, where $J\trianglelefteq \mathcal{L}\times\mathcal{L}$ and $K\subseteq J$. If there is $(a,b)\in J$ with $a\notin\mathrm{M}(\mathcal{L})$, then, as we have seen before, $\mathcal{M}(\mathcal{L})\times\{0\}\subseteq[\mathcal{L}\times\mathcal{L},(a,b)]\subseteq J$. In particular, $\mathrm{M}(\mathcal{L})\times\mathrm{M}(\mathcal{L})/K\subseteq J'$. Then $\mathcal{L}'$ is monolithic.  Moreover, observe that $\mathcal{L}$ is a subalgebra of $\mathcal{L}'$, therefore
\[\textnormal{Id}(\mathcal{L}\times \mathcal{L})\subseteq \textnormal{Id}(\mathcal{L}')\subseteq \textnormal{Id}(\mathcal{L}),\]
it follows that $\mathcal{L}'$ is not critical.

Note that a similar construction works well if we find a nilpotent $\Omega$-algebra, with at least one non-trivial $n$-ary operation with $n\ge2$, $\mathcal{A}$ with $1$-dimensional annihilator.
\end{example}

\begin{example}\textbf{(Critical and non-isomorphic algebras, generating the same variety.)}
Let $\textbf{k}$ be a field such that \(-1\) is not a square in \(\mathbf{k}\). 
Let \(\mathcal{A}_1=\mathrm{Span}\{x,r,t\}\), where the nonzero products are 
\(r^2=rt=-tr=x\), and let \(\mathcal{A}_2=\mathrm{Span}\{x,y,z\}\), where the nonzero 
products are \(y^2=z^2=yz=-zy=x\). 
It is not difficult to prove that
\[
\mathrm{Id}(\mathcal{A}_1)=\mathrm{Id}(\mathcal{A}_2)
=\langle (x_1x_2)x_3,\; x_1(x_2x_3)\rangle_T.
\]
In addition, each proper section of \(\mathcal{A}_k\) (for \(k\in\{1,2\}\)) satisfies 
the identity \([x,y]=0\). 
Thus both algebras are critical.

Assume now that there exists an algebra isomorphism 
\(\psi:\mathcal{A}_1\to\mathcal{A}_2\). 
Since \(\psi(\mathcal{A}_1^2)=\mathcal{A}_2^2\), we obtain 
\(\psi(x)=\lambda x\) for some \(\lambda\in \textbf{k}\setminus\{0\}\). 
Write
\[
\psi(t)=\alpha x+\beta y+\gamma z.
\]
Since \(\psi\) is injective, at least one of \(\beta,\gamma\) is nonzero. 
Furthermore, from 
\[
0=\psi(t^2)
   =(\alpha x+\beta y+\gamma z)^2
   =(\beta^2+\gamma^2)x,
\]
we obtain the relation \(\beta^2+\gamma^2=0\). Hence $\beta=\gamma=0$, due to our assumption on the base field. 
Thus, although 
\(\mathrm{Id}(\mathcal{A}_1)=\mathrm{Id}(\mathcal{A}_2)\) and both algebras are critical, 
we nevertheless have \(\mathcal{A}_1\not\cong\mathcal{A}_2\).

Note that these algebras can be represented as follows:
\[
\mathcal{A}_1
=\mathrm{Span}\bigl\{
  (e_1+e_2)e_{12}+e_1e_{23},\;
  e_2e_{12}+e_2e_{23},\;
  -e_1e_2e_{13}
 \bigr\}
 \subseteq \mathrm{UT}_3(\textbf{k}),
\]
and
\[
\mathcal{A}_2
=\mathrm{Span}\bigl\{
  (e_1+e_2)e_{12}+e_1e_{23},\;
  (e_1-e_2)e_{12}-e_2e_{23},\;
  -e_1e_2e_{13}
 \bigr\}
 \subseteq \mathrm{UT}_3(\textbf{k}).
\]
This example is based on the analogous constructions for finite groups: the dihedral group of order $8$ and the quaternion group of order $8$, see the comments after Corollary 53.33 in \cite{Hanna}.
\end{example}

\section{Further questions}

It is interesting to study what happens when \(\mathcal{A} \in\mathrm{var}(\mathcal{A}')\) under certain conditions on the algebras \(\mathcal{A}\) and \(\mathcal{A}'\). 
This paper provides the answer when both algebras are finite and \(\mathcal{A}\) is prime. 
In particular, we raise the following question for algebraically closed fields:

\medskip
(1) If \(\mathcal{A}\) is a finite-dimensional prime \(\Omega\)-algebra over an algebraically closed field \(\mathbf{k}\) and \(\mathcal{A} \in \mathrm{var}(\mathcal{A}')\), is it true that \(\mathcal{A}\) is a section of \(\mathcal{A}'\)?

\medskip
It is also intriguing that polynomial identities can uniquely determine an algebra in certain classes (for example, upper triangular matrices endowed with a group grading). 
We ask whether this situation could be derived from our results; this would be the case if there exists a suitable signature for a given algebra that turns it into a prime algebra while preserving the polynomial identities in an appropriate sense.

\medskip
(2) Let \(\Omega \subseteq \Omega'\), and assume that \(\Omega\) contains at least one \(n\)-ary operation with \(n \ge 2\). 
Let \(\mathcal{A}\) be an \(\Omega'\)-algebra. 
If \(\mathcal{A}\) is critical as an \(\Omega\)-algebra, is it true that \(\mathcal{A}\) is critical as an \(\Omega'\)-algebra?

The answer is yes. Indeed, if \(\mathrm{Id}_{\Omega'}(\mathcal{A}) = \mathrm{Id}_{\Omega'}(\mathcal{B})\), then \(\mathrm{Id}_\Omega(\mathcal{A}) = \mathrm{Id}_\Omega(\mathcal{B})\). 
Thus, if \(\mathcal{A}\) were not critical as an \(\Omega'\)-algebra, it would belong to the variety generated by some of its sections. 
Each section is also an \(\Omega\)-algebra, so \(\mathcal{A}\) would belong to the variety of \(\Omega\)-algebras generated by them, which is a contradiction.

\medskip
(3) Let \(\mathcal{A}\) and \(\mathcal{B}\) be \(\Omega'\)-algebras, and assume that \((\mathcal{A},\Omega) \cong (\mathcal{B},\Omega)\) and that they are critical as \(\Omega\)-algebras. 
If \(\mathrm{Id}_{\Omega'}(\mathcal{A}) = \mathrm{Id}_{\Omega'}(\mathcal{B})\), is it true that \(\mathcal{A} \cong \mathcal{B}\) as \(\Omega'\)-algebras?

\medskip
($3'$) The answer to the question in (3) is ``no'' if we do not assume that \(\mathcal{A}\) is critical as an \(\Omega\)-algebra. 
For instance, let \(\mathcal{B} = \mathcal{A} = \mathrm{M}_2 \oplus \mathrm{M}_2 \oplus \mathrm{M}_2\). 
Let \(\Gamma\) and \(\Gamma'\) be non-isomorphic \(G\)-gradings on \(\mathrm{M}_2\), and set 
\[
(\mathcal{A}, \Omega_G) = (\mathrm{M}_2, \Gamma) \oplus (\mathrm{M}_2, \Gamma) \oplus (\mathrm{M}_2, \Gamma'), \quad
(\mathcal{B}, \Omega_G) = (\mathrm{M}_2, \Gamma) \oplus (\mathrm{M}_2, \Gamma') \oplus (\mathrm{M}_2, \Gamma').
\]
Then \((\mathcal{A}, \Omega_G) \not\cong (\mathcal{B}, \Omega_G)\), but \(\mathrm{Id}_{\Omega_G}(\mathcal{A}) = \mathrm{Id}_{\Omega_G}(\mathcal{B})\).

\medskip
(4) Let \(\Omega = \Omega' \cap \Omega''\). Let \(\mathcal{A}\) and \(\mathcal{B}\) be \(\Omega'\)-algebras such that \(\mathrm{Id}_{\Omega'}(\mathcal{A}) = \mathrm{Id}_{\Omega'}(\mathcal{B})\) and their structures as \(\Omega''\)-algebras are isomorphic. 
Is it true that \(\mathrm{Id}_{\Omega' \cup \Omega''}(\mathcal{A}) = \mathrm{Id}_{\Omega' \cup \Omega''}(\mathcal{B})\) if we impose additional conditions on \(\mathcal{A}\) and \(\mathcal{B}\)? 
For instance, one may assume that \(\mathcal{A}\) and \(\mathcal{B}\) are critical as \(\Omega\)-algebras.

\medskip
($4'$) A potential application of the above construction is when there exists \(\Omega''\) such that \(\mathcal{A}\) (and \(\mathcal{B}\)) is simple as an \(\Omega''\)-algebra. 
In this case, one could recover an isomorphism \(\mathcal{A} \to \mathcal{B}\) of \(\Omega'\)-algebras under suitable conditions (finite-dimensionality and either an algebraically closed or finite base field).

\medskip
($4''$) The construction in (3') also provides a counterexample for (4) if we do not assume that \(\mathcal{A}\) is critical. 
For example, let \(\Omega'' = \{m, \mathrm{tr}\}\), where \(m\) is the usual product and \(\mathrm{tr}\) is the trace induced from \(\mathrm{M}_6(\mathbb{\mathbf{k}})\). 
Then \((\mathcal{A}, \Omega'')\) is simple, but it is not possible to have 
\(\mathrm{Id}_{\Omega' \cup \Omega''}(\mathcal{A}) = \mathrm{Id}_{\Omega' \cup \Omega''}(\mathcal{B})\), 
since that would imply \((\mathcal{A}, \Omega') \cong (\mathcal{B}, \Omega')\).

\section*{Acknowledgements}
We sincerely thank Professor A.~Ol'shanski\u{\i} for his valuable discussions and suggestions.

\end{document}